\documentstyle[11pt,leqno]{article}

\include{psfig}
\include{epsf}

\newtheorem{theorem}{Theorem}[section]
\newtheorem{lemma}{Lemma}[section]

\newtheorem{example}{Example}[section]

\catcode`\@=11
\renewcommand{\section}{
         \setcounter{equation}{0}
         \@startsection {section}{1}{\z@}{-3.5ex plus -1ex minus
         -.2ex}{2.3ex plus .2ex}{\normalsize\bf}
}
\renewcommand{\subsection}{
         \@startsection {subsection}{1}{\z@}{-3.5ex plus -1ex minus
         -.2ex}{2.3ex plus .2ex}{\normalsize\bf}
} \catcode`\@=12

\def\reals{{\rm\vrule depth0ex width.4pt\kern-.08em R}}
\def\bbbz{{\mathchoice {\hbox{$\sf\textstyle Z\kern-0.4em Z$}}
{\hbox{$\sf\textstyle Z\kern-0.4em Z$}} {\hbox{$\sf\scriptstyle
Z\kern-0.3em Z$}} {\hbox{$\sf\scriptscriptstyle Z\kern-0.2em Z$}}}}

\newcommand{\nc}{\newcommand}
\nc{\W}{{\bf W}} \nc{\A}{{\bf A}} \nc{\bL}{{\bf L}} \nc{\bH}{{\bf
H}} \nc{\C}{{\cal C}}

\def\eq#1{(\ref{e:#1})}

\def\elabel#1{\label{e:#1}}

\begin{document}
\begin{center}
\Large\bf A New Class of Backward Stochastic Partial Differential
Equations with Jumps and Applications
\end{center}
\begin{center}
\large\bf Wanyang Dai
\end{center}
\begin{center}
\small Department of Mathematics\\
Nanjing University, Nanjing 210093, China\\
Email: nan5lu8@netra.nju.edu.cn\\
Date: 5 May 2011
\end{center}

\vskip 0.1 in
\begin{abstract}
We formulate a new class of stochastic partial differential
equations (SPDEs), named {\em high-order vector backward SPDEs
(B-SPDEs) with jumps}, which allow the high-order integral-partial
differential operators into both drift and diffusion coefficients.
Under certain type of Lipschitz and linear growth conditions, we
develop a method to prove the existence and uniqueness of adapted
solution to these B-SPDEs with jumps. Comparing with the existing
discussions on conventional backward stochastic (ordinary)
differential equations (BSDEs), we need to handle the
differentiability of adapted triplet solution to the B-SPDEs with
jumps, which is a subtle part in justifying our main results due to
the inconsistency of differential orders on two sides of the B-SPDEs
and the partial differential operator appeared in the diffusion
coefficient. In addition, we also address the issue about the
B-SPDEs under certain Markovian random environment and employ a
B-SPDE with strongly nonlinear partial differential operator in the
drift coefficient to illustrate the usage of our main results in
finance.\\

\end{abstract}

\section{Introduction}

Motivated from mean-variance hedging (see, e.g.,
Dai~\cite{dai:meavar}) and utility based optimal portfolio choice
(see, e.g., Becherer~\cite{bec:bousol}, Musiela and
Zariphopoulou~\cite{muszar:stopar}) in finance, and multi-channel
(or multi-valued) image regularization such as color images in
computer vision and network application (see, e.g., Caselles {\em et
al.}~\cite{cassap:vecmed}, Tschumperl\'e and
Deriche~\cite{tscder:regort,tscder:conunc,tscder:anidif}, and
references therein), we formulate a new class of SPDEs, named {\em
high-order vector B-SPDEs with jumps}, which allow high-order
integral-partial differential operators ${\cal L}$ and ${\cal J}$
into both drift and diffusion coefficients as shown in the following
equation \eq{bspdef},
\begin{eqnarray}
\;\;\;\;\;V(t,x)&=&H(x)+\int_{t}^{T}{\cal L}(s^{-},x,V,\cdot)ds
+\int_{t}^{T}\left({\cal J}(s^{-},x,V,\cdot)
-\bar{V}(s^{-},x)\right)dW(s)
\elabel{bspdef}\\
&&-\int_{t}^{T}\int_{z>0}\tilde{V}(s^{-},x,z,\cdot)
\tilde{N}(\lambda ds,x,dz). \nonumber
\end{eqnarray}
where the operator ${\cal L}$ depends not only on
$V,\bar{V},\tilde{V}$ but also on their associated partial
derivatives, i.e., for each integer $k\geq 2$ and $m\geq 0$, ${\cal
L}$ and ${\cal J}$ are defined by
\begin{eqnarray}
{\cal L}(s,x,V,\cdot)&\equiv&{\cal
L}(s,x,V(s,x),V^{(k)}(s,x),\bar{V}(s,x),
\bar{V}^{(m)}(s,x),\tilde{V}(s,x),\cdot),\nonumber\\
{\cal J}(s,x,V,\cdot)&=&({\cal J}_{1}(s,x,V,\cdot),...,{\cal
J}_{d}(s,x,V,\cdot)),
\nonumber\\
{\cal J}_{i}(s,x,V,\cdot)&\equiv&{\cal
J}_{i}(s,x,V(s,x),V^{(k)}(s,x),\cdot),\;\;\;\;\;\;i\in\{1,...,d\}.
\nonumber
\end{eqnarray}
Under certain type of Lipschitz and linear growth conditions, we
prove the existence and uniqueness of adapted triplet solution
$(V,\bar{V},\tilde{V})$ to these B-SPDEs. When the partial
differential operator ${\cal L}$ depends only on $x,V,\bar{V}$, and
$\tilde{V}$ but not on their associated derivatives and ${\cal
L}=0$, our B-SPDEs with jumps reduce to conventional BSDEs with
jumps (see, e.g., Becherer~\cite{bec:bousol}, Dai~\cite{dai:meavar},
Tang and Li~\cite{tanli:neccon}).

BSDEs were first introduced by Bismut~\cite{bis:concon} and the
first result for the existence of an adapted solution to a
continuous nonlinear BSDE was obtained by Pardoux and
Peng~\cite{parpen:adasol}. Since then, numerous extensions along the
line have been conducted, such as, Tang and Li~\cite{tanli:neccon}
get the first adapted solution to a BSDE with Poisson jumps for a
fixed terminal time and Situ~\cite{sit:solbac} extended the result
to the case where the BSDE is with bounded random stopping time as
its terminal time and non-Lipschitz coefficient. Currently, BSDEs
are still an active area of research in both theory and
applications, see, e.g., Becherer~\cite{bec:bousol}, Cohen and
Elliott~\cite{cohell:genthe}, Cr\'epey and
Matoussi~\cite{cremat:refdou}, Dai~\cite{dai:meavar}, Lepeltier {\em
et al.}~\cite{lepmat:refbac}, Yin and Mao~\cite{yinmao:adasol}, and
references therein.

The study on SPDEs receives a great attention recently (see, e.g.,
Pardoux~\cite{par:stopar} and Hairer~\cite{hai:intsto}).
Particularly, Pardoux and Peng~\cite{parpen:bacdou} introduces a
system of semi-linear parabolic SPDE in a backward manner and
establish the existence and uniqueness of adapted solution to the
SPDE under smoothness assumptions on the coefficients, and moreover,
the authors in \cite{parpen:bacdou} also employ backward doubly SDEs
(BDSDE) to provide a probabilistic representation for the parabolic
SDE. Since then, numerous researches have been conducted in terms of
weak solution and stationary solution to the semi-linear SPDE (see,
e.g., Bally and Matoussi~\cite{balmat:weasol}, Zhang and
Zhao~\cite{zhazha:stasol}, and references therein). However, our
B-SPDEs exhibited in \eq{bspdef} are fundamentally different from
the SPDEs as introduced in Pardoux and Peng~\cite{parpen:bacdou} and
as studied in most of the existing researches in the following
aspects: First, our system formulation is a direct generalization of
the conventional BSDEs, i.e., both the drift and diffusion
coefficients of our B-SPDEs depend on the triplet
$(V,\bar{V},\tilde{V})$ and its associated partial derivatives not
just on $V$ and its associated partial derivatives; Second, our
B-SPDEs are based on high-order partial derivatives and are subject
to jumps. One special case of our B-SPDEs available in the
literature is the one derived in Musiela and
Zariphopoulou~\cite{muszar:stopar} for the purpose of
optimal-utility based portfolio choice, which is strongly nonlinear
in the sense that is addressed in Lions and
Souganidis~\cite{liosou:notaux}.

Note that the B-SPDEs presented in \eq{bspdef} are vector B-SPDEs
with jumps, which are motivated from various aspects such as
multi-channel image regularization in computer vision and network
application through vector PDEs (see, e.g., Caselles {\em et
al.}~\cite{cassap:vecmed}, Tschumperl\'e and
Deriche~\cite{tscder:regort,tscder:conunc,tscder:anidif}, and
references therein), coupling and synchronization in random dynamic
systems through vector SPDEs (see, e.g., Mueller~\cite{mue:couinv},
Chueshov and Schmalfu$\ss$~\cite{chusch:massla}, and references
therein).

To show our formulated system well-posed, we develop a method based
on a scheme used for conventional BSDEs (see, e.g., Yong and
Zhou~\cite{yonzho:stocon}) to prove the existence and uniqueness of
adapted solution to our B-SPDEs with jumps in \eq{bspdef} under
certain Lipschitz and linear growth conditions. One fundamental
issue we need to handle in the method is the differentiability of
the triplet solution to our B-SPDEs with jumps, which is a subtle
part in the analysis due to the inconsistency of differential orders
on two sides of the B-SPDEs and the partial differential operators
appeared in the diffusion coefficient. So more involved functional
spaces and techniques are required. In addition, although there is
no perfect theory in dealing with the strongly nonlinear SPDEs (see,
e.g., Pardoux~\cite{par:stopar}), our discussions about the adapted
solution to \eq{bspdef} can provide some reasonable interpretation
concerning the unique existence of adapted solution before a random
bankruptcy time to the strongly nonlinear B-SPDE derived in Musiela
and Zariphopoulou~\cite{muszar:stopar}.

In the paper, we also provide some discussion concerning our B-SPDEs
under random environment, e.g., the variable $x$ in \eq{bspdef} is
replaced by a continuous Markovian process $X(\cdot)$. To be
convenient for readers, we present a rough graph in
Figure~\ref{twouserregion} with respect to sample surfaces for a
solution to a B-SPDE and in terms of sample curves for a solution to
the B-SPDE under random environment.
\begin{figure}[tbh]
\centerline{\epsfxsize=3.0in\epsfbox{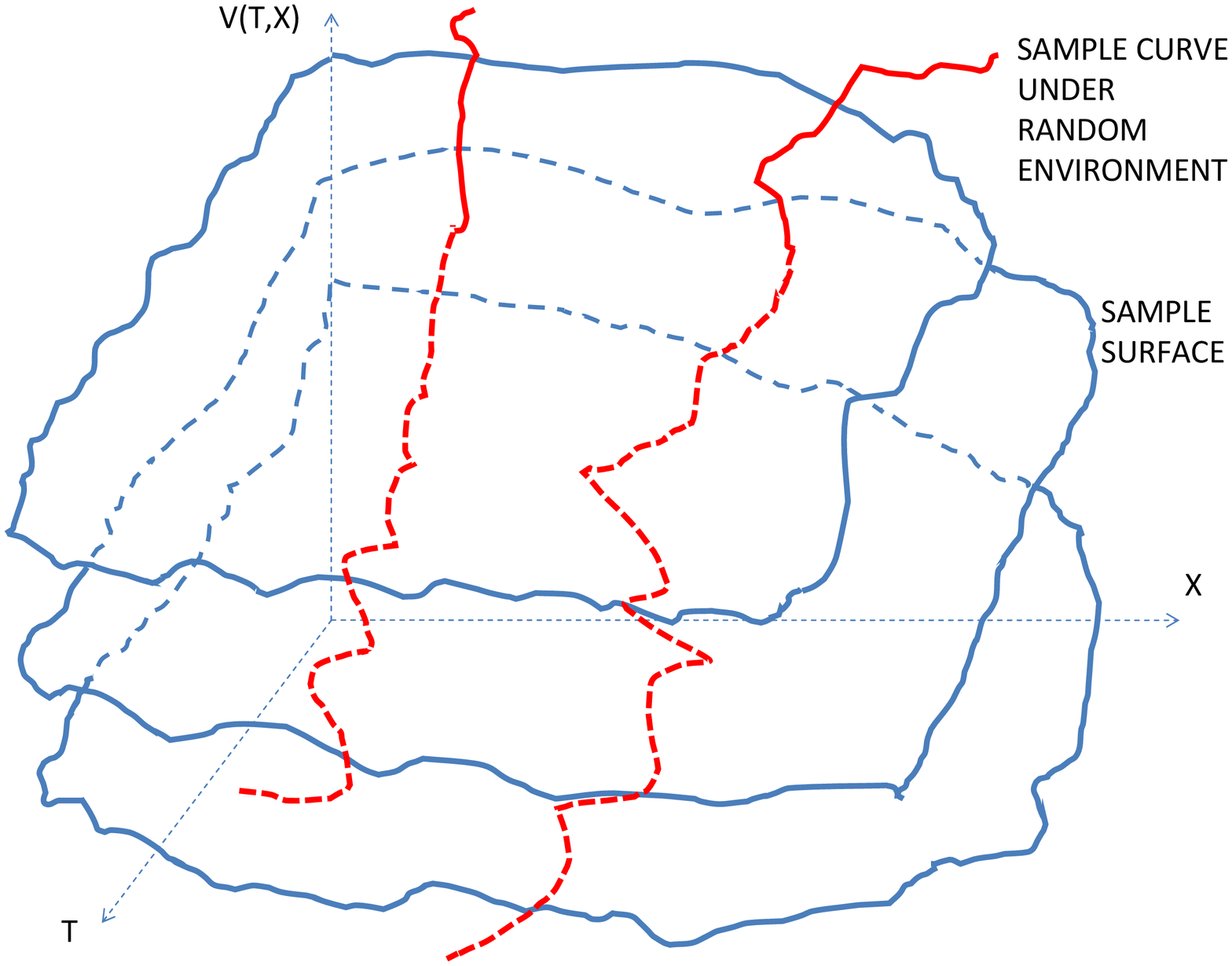}} \caption{\small
Sample surfaces and sample paths for a B-SPDE and a B-SPDE under
random environment} \label{twouserregion}
\end{figure}

The rest of the paper is organized as follows. In
Section~\ref{finited}, we first introduce a class of B-SPDEs with
jumps in finite space domain, then we state and prove our main
theorem. In Section~\ref{infran}, we extend our discussions in the
previous section to the case corresponding to infinite space domain
and under random environment. In Section~\ref{expm}, we use an
example to illustrate the usage of our main results in finance.

\section{A Class of B-SPDEs with Jumps in Finite Space Domain}
\label{finited}

\subsection{Required Probability and Functional
Spaces}

First of all, we introduce some notations to be used in the paper.
Let $(\Omega,{\cal F},P)$ be a fixed complete probability space on
which are defined a standard $d$-dimensional Brownian motion
$W\equiv\{W(t),t\in[0,T]\}$ with $W(t)=(W_{1}(t),...,W_{d}(t))'$ and
$h-$dimensional subordinator $L\equiv\{L(t),t\in[0,T]\}$ with
$L(t)\equiv(L_{1}(t),...,L_{h}(t))'$ and c\`adl\`ag sample paths for
some fixed $T\in[0,\infty)$ (see, e.g., Applebaum~\cite{app:levpro},
Bertoin~\cite{ber:levpro}, and Sato~\cite{sat:levpro} for more
details about subordinators and L\'evy processes), where the prime
denotes the corresponding transpose of a matrix or a vector.
Moreover, $W$, $L$ and their components are assumed to be
independent of each other. In addition, each subordinator $L_{i}$
with $i\in\{1,...,h\}$ can be represented by (see, e.g., Theorem
13.4 and Corollary 13.7 in Kallenberg~\cite{kal:foumod})
\begin{eqnarray}
&&L_{i}(t)=a_{i}t+\int_{(0,t]}\int_{z_{i}>0}z_{i}N_{i}(ds,dz_{i}),
\;t\geq 0 \elabel{subordrep}
\end{eqnarray}
where $N_{i}((0,t]\times A)\equiv\sum_{0<s\leq t}
I_{A}(L(s)-L_{i}(s^{-}))$ denotes a Poisson random measure with a
deterministic, time-homogeneous intensity measure
$ds\nu_{i}(dz_{i})$, where $I_{A}(\cdot)$ is the index function over
the set $A$, the constant $a_{i}$ is taken to be zero, and $\nu_{i}$
is the L\'{e}vy measure. Related to the probability space
$(\Omega,{\cal F},P)$, we suppose that there is a filtration
$\{{\cal F}_{t}\}_{t\geq 0}$ with ${\cal
F}_{t}\equiv\sigma\{W(s),L(\lambda s): 0\leq s\leq t\}$ for each
$t\in[0,T]$, $\lambda=(\lambda_{1},...\lambda_{h})'>0$, and
$L(\lambda s)=(L_{1}(\lambda_{1}s),...,L_{h}(\lambda_{h}s))'$.

Secondly, let ${\cal N}=\{1,2,...,\}$ and $D$ be a close connected
domain in $R^{p}$ for a given $p\in{\cal N}$. Then we can use
$C^{k}(D,R^{q})$ for each $k,p,q\in{\cal N}$ to denote the Banach
space of all functions $f$ having continuous derivatives up to the
order $k$ with the following uniform norm,
\begin{eqnarray}
&&\|f\|_{C^{k}(D,q)}=\max_{c\in\{1,...,k\}}\max_{j\in\{1,...,r(c)\}}
\sup_{x\in D}\left|f^{(c)}_{j}(x)\right| \elabel{fckupd}
\end{eqnarray}
for each $f\in C^{k}(D,R^{q})$, where $r(c)$ for each
$c\in\{0,1,...,k\}$ is the total number of the following partial
derivatives of the order $c$
\begin{eqnarray}
&&f^{(c)}_{r,(i_{1}...i_{p})}(x)=\frac{\partial^{c}f_{r}(x)}{\partial
x_{1}^{i_{1}}...\partial x_{p}^{i_{p}}} \elabel{difoperatorI}
\end{eqnarray}
with $i_{l}\in\{0,1,...,c\}$, $l\in\{1,...,p\}$, $r\in\{1,...,q\}$,
and $i_{1}+...+i_{p}=c$. Moreover, for the late purpose, we let
\begin{eqnarray}
f_{(i_{1},...,i_{p})}^{(c)}&\equiv&(f_{1,(i_{1},...,i_{p})}^{(c)},...,
f_{q,(i_{1},...,i_{p})}^{(c)}),
\elabel{difoperatoro}\\
f^{(c)}(x)&\equiv&(f_{1}^{(c)}(x),...,f_{r(c)}^{(c)}(x)),
\elabel{difoperator}
\end{eqnarray}
where each $j\in\{1,...,r(c)\}$ corresponds to a $p$-tuple
$(i_{1},...,i_{p})$ and a $r\in\{1,...,q\}$. In addition, let
$C^{\infty}(D,R^{q})$ denote the following Banach space, i.e.,
\begin{eqnarray}
&&C^{\infty}(D,R^{q})\equiv
\left\{f\in\bigcap_{k=1}^{\infty}C^{k}(D,R^{q}),
\|f\|_{C^{\infty}(D,q)}<\infty\right\} \elabel{cinfity}
\end{eqnarray}
where
\begin{eqnarray}
&&\|f\|^{2}_{C^{\infty}(D,q)}=\sum_{k=1}^{\infty}\xi(k)
\|f\|^{2}_{C^{k}(D,q)} \elabel{inftynorm}
\end{eqnarray}
for some discrete function with respect to $k\in\{0,1,2,...\}$,
which is fast decaying in $k$. For convenience, we take
$\xi(k)=e^{-k}$.

Thirdly, we introduce some measurable spaces to be used in the
sequel. Let $L^{2}_{{\cal F}}([0,T];R^{q})$ denote the set of all
$R^{q}$-valued measurable stochastic processes $Z(t,x)$ adapted to
$\{{\cal F}_{t},t\in[0,T]\}$ for each $x\in D$, which are in
$C^{\infty}(D,R^{q})$ for each fixed $t\in[0,T]$), such that
\begin{eqnarray}
&&E\left[\int_{0}^{T}\|Z(t)\|^{2}_{C^{\infty}(D,q)}dt\right]<\infty
\elabel{adaptednormI}
\end{eqnarray}
and let $L^{2}_{{\cal F},p}([0,T],R^{q})$ denote the corresponding
set of predictable processes (see, e.g., Definition 5.2 and
Definition 1.1 respectively in pages 21 and 45 of Ikeda and Watanabe
\cite{ikewat:stodif}). Moreover, let $L^{2}_{{\cal
F}_{T}}(\Omega;R^{q})$ denote the set of all $R^{q}$-valued, ${\cal
F}_{T}$-measurable random variables $\xi(x)$ for each $x\in D$,
where $\xi(x)\in C^{\infty}(D,R^{q})$ satisfies
\begin{eqnarray}
&&E\left[\|\xi\|^{2}_{C^{\infty}(D,q)}\right]<\infty.
\elabel{initialx}
\end{eqnarray}
In addition, let $L^{2}_{p}([0,T],$ $R^{h})$ be the set of all
$R^{h}$-valued predictable processes $\tilde{V}(t,x,z)=$
$(\tilde{V}_{1}(t,x,z),$ $...,$ $\tilde{V}_{h}(t,x,z))'$ for each
$x\in D$ and $z\in R_{+}^{h}$, satisfying
\begin{eqnarray}
&&E\left[\sum_{i=1}^{h}\int_{0}^{T}\int_{z_{i}>0}
\left\|\tilde{V}_{i}(t^{-},z)\right\|^{2}_{C^{\infty}(D,q)}
\nu_{i}(dz_{i})dt\right]<\infty \elabel{normsI}
\end{eqnarray}
and let
\begin{eqnarray}
&&L^{2}_{\nu,c}(D\times R^{h}_{+},R^{q\times
h})\equiv\left\{\tilde{v}: D\times R^{h}_{+}\rightarrow R^{q\times
h},\sum_{i=1}^{h}\int_{z_{i}>0}
\left\|\tilde{v}_{i}(z_{i})\right\|^{2}_{C^{c}(D,q)}
\nu_{i}(dz_{i})<\infty\right\} \elabel{lvrqh}
\end{eqnarray}
with the associated norm for any $\tilde{v}\in L^{2}_{\nu,c}(D\times
R^{h}_{+},R^{q\times h})$ and $c\in\{0,1,...,\infty\}$ as follows,
\begin{eqnarray}
&&\|\tilde{v}\|_{\nu,c}\equiv\left(\sum_{i=1}^{h}\int_{z_{i}>0}
\left\|\tilde{v}_{i}(z_{i})\right\|^{2}_{C^{c}(D,q)}
\lambda_{i}\nu_{i}(dz_{i})\right)^{\frac{1}{2}}. \elabel{vnorm}
\end{eqnarray}
In the end, we define
\begin{eqnarray}
&&{\cal Q}^{2}_{{\cal F}}([0,T])\equiv L^{2}_{{\cal
F}}([0,T],R^{q})\times L^{2}_{{\cal F},p}([0,T],R^{qd})\times
L^{2}_{p}([0,T],R^{q\times h}). \elabel{bunisoI}
\end{eqnarray}

\subsection{The B-SPDEs}

First of all, we introduce a class of $q$-dimensional B-SPDEs with
jumps and terminal random variable $H(x)\in L^{2}_{{\cal
F}_{T}}(\Omega;R^{q})$ for each $x\in D$ as presented in
\eq{bspdef}, where for each $s\in[0,T]$ and $z=(z_{1},...,z_{h})\in
R^{h}_{+}$,
\begin{eqnarray}
\bar{V}(s,\cdot)&=&\left(\bar{V}_{1}(s,\cdot),...,\bar{V}_{d}(s,\cdot)
\right)\in C^{\infty}(D,R^{q\times
d}),\nonumber\\
\tilde{V}(s,\cdot,z)&=&(\tilde{V}_{1}(s,\cdot,z_{1}),...,
\tilde{V}_{h}(s,\cdot,z_{h}))\in C^{\infty}(D,R^{q\times h}),
\nonumber\\
\tilde{N}(\lambda
ds,x,dz)&=&(\tilde{N}_{1}(\lambda_{1}ds,x,dz_{1}),...,
\tilde{N}_{h}(\lambda_{h}ds,x,dz_{h}))'. \nonumber
\end{eqnarray}
Moreover, in \eq{bspdef}, ${\cal L}$ is a $q$-dimensional
integral-partial differential operator satisfying, a.s.,
\begin{eqnarray}
&&\left\|\Delta{\cal L}^{(c)}(s,x,u,v)\right\|\leq
K_{D}\left(\|u-v\|_{C^{k+c}(D,q)}+\|\bar{u}-\bar{v}\|_{C^{m+c}(D,qd)}
+\|\tilde{u}-\tilde{v}\|_{\nu,c}\right) \elabel{blipschitz}
\end{eqnarray}
for any $(u,\bar{u},\tilde{u})$, $(v,\bar{v},\tilde{v})$ $\in
C^{k}(D,R^{q})\times C^{m}(D,R^{q\times d})\times
L^{2}_{\nu,c}(R^{h}_{+},R^{q\times h})$ with
$c\in\{0,1,...,\infty\}$, where $K_{D}$ depending on the domain $D$
is a nonnegative constant, $\|A\|$ is the largest absolute value of
entries (or components) of the given matrix (or vector) $A$, and
\begin{eqnarray}
&&\Delta{\cal L}^{(c)}(s,x,u,v) \equiv{\cal L}^{(c)}(s,x,u,
\cdot)-{\cal L}^{(c)}(s,x,v,\cdot).
\end{eqnarray}
Similarly, ${\cal J}=({\cal J}_{1},...,{\cal J}_{d})$ is a $q\times
d$-dimensional partial differential operator satisfying, a.s.,
\begin{eqnarray}
&&\|\Delta{\cal J}^{(c)}(s,x,u,v)\|\leq
K_{D}\left(\|u-v\|_{C^{m+c}(D,q)}\right). \elabel{blipschitzo}
\end{eqnarray}
Moreover, we suppose that
\begin{eqnarray}
\left\|{\cal L}^{(c)}(s,x,u,\cdot)\right\|&\leq&
K_{D}\left(\|u\|_{C^{k+c}(D,q)}+\|\bar{v}\|_{C^{m+c}(D,qd)}
+\left\|\tilde{v}\right\|_{\nu,c}\right),
\elabel{blipschitzo}\\
\left\|{\cal J}^{(c)}(s,x,u,\cdot)\right\|&\leq&
K_{D}\|u\|_{C^{m+c}(D,q)}. \elabel{blipschitzI}
\end{eqnarray}
\begin{example}
The following conventional linear partial differential operators
satisfy the conditions as stated in
\eq{blipschitz}-\eq{blipschitzI},
\begin{eqnarray}
({\cal
L}u)(t,x)&=&\sum_{i,j=1}^{p}a_{ij}(x)\frac{\partial^{2}u(t,x)}{\partial
x_{i}{\partial x_{j}}}+\sum_{j=1}^{d}b_{j}(x)\frac{\partial
u(t,x)}{\partial x_{j}}+c(x)u(t,x)
\nonumber\\
&&+\sum_{i,j=1}^{d}\bar{a}_{ij}(x)\frac{\partial^{2}\bar{u}(t,x)}{\partial
x_{i}{\partial x_{j}}}+\sum_{j=1}^{d}\bar{b}_{j}(x)\frac{\partial
\bar{u}(t,x)}{\partial x_{j}}+\bar{c}(x)\bar{u}(t,x) \nonumber\\
({\cal
J}u)(t,x)&=&\sum_{i,j=1}^{p}a_{ij}(x)\frac{\partial^{2}u(t,x)}{\partial
x_{i}{\partial x_{j}}}+\sum_{j=1}^{d}b_{j}(x)\frac{\partial
u(t,x)}{\partial x_{j}}+c(x)u(t,x), \nonumber
\end{eqnarray}
where $a_{ij}^{(c)}(x)$, $b_{j}^{(c)}(x)$, $c^{(c)}(x)$,
$\bar{a}_{ij}^{(c)}(x)$, $\bar{b}_{j}^{(c)}(x)$, and
$\bar{c}^{(c)}(x)$ are uniformly bounded over all $x\in D$ and
$i,j\in\{1,...,d\}$ and $c\in\{0,1,2,....\}$.
\end{example}
\begin{theorem}\label{bsdeyI}
Under conditions of \eq{blipschitz}-\eq{blipschitzI}, if ${\cal
L}(t,x,v,\cdot)$ and ${\cal J}(t,x,v,\cdot)$ are $\{{\cal
F}_{t}\}$-adapted for each fixed $x\in D$ and any given
$(v,\bar{v},\tilde{v})\in C^{\infty}(D,R^{q})\times
C^{\infty}(D,R^{q\times d})\times L^{2}_{\nu,\infty}(D\times
R^{h}_{+},R^{q\times h})$ with
\begin{eqnarray}
&&{\cal L}(\cdot,x,0,\cdot),{\cal J}(\cdot,x,0,\cdot)\in
L^{2}_{{\cal F}}\left([0,T],R^{q}\right), \elabel{blipic}
\end{eqnarray}
then the B-SPDE \eq{bspdef} has a unique adapted solution
satisfying, for each $x\in D$ and $z\in R_{+}^{h}$,
\begin{eqnarray}
&&(V(\cdot,x),\bar{V}(\cdot,x),\tilde{V}(\cdot,x,z))\in {\cal
Q}^{2}_{{\cal F}}([0,T]) \elabel{buniso}
\end{eqnarray}
where $V$ is a c\`adl\`ag process and the uniqueness is in the
sense: if there exists another solution
$(U(t,x),\bar{U}(t,x),\tilde{U}(t,x,z))$ as required, we have
\begin{eqnarray}
&&E\left[\int_{0}^{T}\left(\|U(t)-V(t)\|^{2}_{C^{\infty}(D,q)}+
\|\bar{U}(t)-\bar{V}(t)\|^{2}_{C^{\infty}(D,qd)}+\|\tilde{U}(t)
-\tilde{V}(t)\|_{\nu,\infty}^{2}\right)dt\right]=0.
\nonumber
\end{eqnarray}
\end{theorem}

We divide the proof of Theorem~\ref{bsdeyI} into the following three
lemmas.
\begin{lemma}\label{martindecom}
Under the conditions of Theorem~\ref{bsdeyI}, for each fixed $x\in
D$, $z\in R_{+}^{h}$, and a triplet
\begin{eqnarray}
&&(U(\cdot,x),\bar{U}(\cdot,x),\tilde{U}(\cdot,x,z))\in {\cal
Q}^{2}_{{\cal F}}([0,T]),\elabel{threeu}
\end{eqnarray}
there exists another triplet
$(V(\cdot,x),\bar{V}(\cdot,x),\tilde{V}(\cdot,x,z)$ such that
\begin{eqnarray}
\;\;\;\;V(t,x)&=&H(x)+\int_{t}^{T}{\cal
L}(s^{-},x,U,\cdot)ds+\int_{t}^{T}\left({\cal J}(s^{-},x,U,\cdot)
-\bar{V}(s^{-},x)\right)dW(s)
\elabel{sigmanV}\\
&&-\int_{t}^{T}\int_{z>0}\tilde{V}(s^{-},x,z) \tilde{N}(\lambda
ds,x,dz), \nonumber
\end{eqnarray}
where $V$ is a $\{{\cal F}_{t}\}$-adapted c\`adl\`ag process,
$\bar{V}$ and $\tilde{V}$ are the corresponding predictable
processes, and for each $x\in D$,
\begin{eqnarray}
&&E\left[\int_{0}^{T}\|V(t,x)\|^{2}dt\right]<\infty,
\elabel{maxnormI}\\
&&E\left[\int_{0}^{T}\|\bar{V}(t,x)\|^{2}dt\right]<\infty,
\elabel{maxnormII}\\
&&E\left[\sum_{i=1}^{h}\int_{0}^{T}\int_{z_{i}>0}
\left\|\tilde{V}_{i}(t^{-},x,z)\right\|^{2}
\nu_{i}(dz_{i})dt\right]<\infty. \elabel{maxnormIII}
\end{eqnarray}
\end{lemma}
\begin{proof}
First of all, for each fixed $x\in D$, $z\in R_{+}^{h}$, and a
triplet $(U(\cdot,x),$ $\bar{U}(\cdot,x),$ $\tilde{U}(\cdot,x,z))$
as stated in \eq{threeu}, it follows from conditions
\eq{blipschitz}-\eq{blipic} that
\begin{eqnarray}
&&{\cal L}(\cdot,x,U,\cdot)\in L^{2}_{{\cal
F}}\left([0,T],R^{q}\right),\;{\cal J}(\cdot,x,U,\cdot)\in
L^{2}_{{\cal F}}([0,T],R^{q\times d}). \elabel{ladaptedI}
\end{eqnarray}
Now consider ${\cal L}$ and ${\cal J}$ in \eq{ladaptedI} as two new
starting ${\cal L}(\cdot,x,0,\cdot)$ and ${\cal
J}(\cdot,x,0,\cdot)$, then it follows from the Martingale
representation theorem (see, e.g., Lemma 2.3 in Tang and
Li~\cite{tanli:neccon}) that there exists a unique pair of
predictable processes $(\bar{V}(\cdot,x),\tilde{V}(\cdot,x,z))$
which are square-integrable for each $x\in D$ in the senses of
\eq{maxnormII}-\eq{maxnormIII} such that
\begin{eqnarray}
\;\;\;\;\hat{V}(t,x)&\equiv& \left.E\left[H(x) +\int_{0}^{T}{\cal
L}(s^{-},x,U,\cdot)ds+\int_{0}^{T}{\cal
J}(s^{-},x,U,\cdot)dW(s)\right|{\cal F}_{t}\right]
\elabel{sigmanI}\\
&=&\hat{V}(0,x)+\int_{0}^{t}\bar{V}(s^{-},x)dW(s) +\int_{0}^{t}
\int_{z>0}\tilde{V}(s^{-},x,z) \tilde{N}(\lambda ds,x,dz)
\nonumber
\end{eqnarray}
which implies that
\begin{eqnarray}
\;\;\;\hat{V}(0,x)&=&H(x)+\int_{0}^{T}{\cal L}(s^{-},x,U,\cdot)ds
+\int_{0}^{T}{\cal J}(s^{-},x,U,\cdot)dW(s)
\elabel{sigmanII}\\
&&-\int_{0}^{T}\bar{V}(s^{-},x)dW(s) -\int_{0}^{T}\int_{z>0}
\tilde{V}(s^{-},x,z)\tilde{N}(\lambda ds,x,dz). \nonumber
\end{eqnarray}
Moreover, due to the Corollary in page 8 of
Protter~\cite{pro:stoint}, $\hat{V}(\cdot,x)$ can be taken as a
c\`adl\`ag process. Now we define a process $V$ as follows,
\begin{eqnarray}
V(t,x)&\equiv& E\left[H(x) +\int_{t}^{T}{\cal
L}(s^{-},x,U,\cdot)ds+\left.\int_{t}^{T}{\cal
J}(s^{-},x,U,\cdot)dW(s) \right|{\cal F}_{t}\right]
\elabel{sigmanIII}
\end{eqnarray}
Then by simple calculation, we know that $V(\cdot,x)$ is
square-integrable in the sense of \eq{maxnormI}, and moreover, it
follows from \eq{sigmanI}-\eq{sigmanIII} that
\begin{eqnarray}
V(t,x)&=&\hat{V}(t,x)-\int_{0}^{t}{\cal L}(s^{-},x,U,\cdot)ds
-\int_{0}^{t}{\cal J}(s^{-},x,U,\cdot)dW(s) \elabel{sigmanIV}
\end{eqnarray}
which indicates that $V(\cdot,x)$ is a c\`adl\`ag process.
Furthermore, for a given triplet $(U(\cdot,x),$ $\dot{U}(\cdot,x)$,
$\bar{U}(\cdot,x),$ $\tilde{U}(\cdot,x,z))$, it follows from
\eq{sigmanI}-\eq{sigmanII} and \eq{sigmanIV} that the corresponding
triplet $(V(\cdot,x),$ $\bar{V}(\cdot,x),$ $\tilde{V}(\cdot,x,z))$
satisfies the equation \eq{sigmanV} as stated in the lemma,  which
also implies that
\begin{eqnarray}
\;\;\;\;\;\;V(t,x)&\equiv&V(0,x)-\int_{0}^{t}{\cal
L}(s^{-},x,U,\cdot)ds -\int_{0}^{t}\left({\cal J}(s^{-},x,U,\cdot)
-\bar{V}(s^{-},x)\right)dW(s)
\elabel{sigmanVI}\\
&&+\int_{0}^{t}\int_{z>0}\tilde{V}(s^{-},x,z) \tilde{N}(\lambda
ds,x,dz). \nonumber
\end{eqnarray}
Hence we complete the proof of Lemma~\ref{martindecom}. $\Box$
\end{proof}
\begin{lemma}\label{differentiableV}
Under the conditions of Theorem~\ref{bsdeyI}, for each fixed $x\in
D$, $z\in R_{+}^{h}$, and a triplet as in \eq{threeu}, we define
$V(t,x),$ $\bar{V}(t,x),$ $\tilde{V}(t,x,z)$ through \eq{sigmanV}.
Then $(V^{(c)}(\cdot,x)$, $\bar{V}^{(c)}(\cdot,x)$,
$\tilde{V}^{(c)}(\cdot,x,z))$ for each $c\in\{0,1,...,\}$ exists
a.s. and satisfies a.s.
\begin{eqnarray}
V^{(c)}_{(i_{1}...i_{p})}(t,x)&=&H^{(c)}_{(i_{1}...i_{p})}(x)
+\int_{t}^{T}{\cal L}^{(c)}_{(i_{1}...i_{p})}(s^{-},x,U,\cdot)ds
\elabel{pdsigmanV}\\
&&+\int_{t}^{T}\left({\cal
J}^{(c)}_{(i_{1}...i_{p})}(s^{-},x,U,\cdot)
-\bar{V}^{(c)}_{(i_{1}...i_{p})}(s^{-},x)\right)dW_{i}(s)
\nonumber\\
&&-\int_{t}^{T}\int_{z>0}\tilde{V}^{(c)}_{(i_{1}...i_{p})}
(s^{-},x,z)\tilde{N}(\lambda ds,x,dz), \nonumber
\end{eqnarray}
where $i_{1}+...+i_{p}=c$ and $i_{l}\in\{0,1,...,c\}$ with
$l\in\{1,...,p\}$. Moreover,  $V^{(c)}_{(i_{1}...i_{p})}$ for each
$c\in\{0,1,...\}$ is a $\{{\cal F}_{t}\}$-adapted c\`adl\`ag
process, $\bar{V}^{(c)}_{(i_{1}...i_{p})}$ and
$\tilde{V}^{(c)}_{(i_{1}...i_{p})}$ are the corresponding
predictable processes, which are square-integrable in the senses of
\eq{maxnormI}-\eq{maxnormIII}.
\end{lemma}
\begin{proof}
First of all, we show that the claim in the lemma is true for $c=1$.
To do so, for each given $t\in[0,T],x\in D,z\in R_{+}^{h}$ and
$(U(t,x),\bar{U}(t,x),\tilde{U}(t,x,z)$ as in the lemma, let
\begin{eqnarray}
&&(V^{(1)}_{(l)}(t,x),\bar{V}^{(1)}_{(l)}(t,x),
\tilde{V}^{(1)}_{(l)}(t,x,z))
\elabel{tripletdv}
\end{eqnarray}
be defined through \eq{sigmanV} where ${\cal L}$ and ${\cal J}$ are
replaced by their first-order partial derivatives ${\cal
L}^{(1)}_{(l)}$ and ${\cal J}^{(1)}_{(l)}$ in terms of $x_{l}$ with
$l\in\{1,...,p\}$. Then we can show that the triplet defined in
\eq{tripletdv} for each $l$ is indeed the required first-order
partial derivative of $(V,\bar{V},\tilde{V})$ that is defined
through \eq{sigmanV} for the given $(U,\bar{U},\tilde{U})$.

As a matter of fact, for each
$f\in\{U,\bar{U},\tilde{U},V,\bar{V},\tilde{V},\tilde{N}\}$, small
enough positive constant $\delta$, and $l\in\{1,...,p\}$, define
\begin{eqnarray}
&&f_{(l),\delta}(t,x)\equiv f(t,x+\delta e_{l}),\elabel{fdeltan}
\end{eqnarray}
where $e_{l}$ is the unit vector whose $l$th component is one and
others are zero. Moreover, let
\begin{eqnarray}
&&\Delta f^{(1)}_{(l),\delta}(t,x)
=\frac{f_{(l),\delta}(t,x)-f(t,x)}{\delta}-f^{(1)}_{(l)}(t,x)
\elabel{partialII}
\end{eqnarray}
for each $f\in\{U,\bar{U},\tilde{U},V,\bar{V},\tilde{V}\}$. In
addition, let
\begin{eqnarray}
\Delta{\cal I}^{(1)}_{(l),\delta}(s,x,U)
&=&\frac{1}{\delta}\left({\cal I}(s,x+\delta e_{l},U(s,x+\delta
e_{l}),\cdot)-{\cal I}(s,x,U(s,x),\cdot)\right)
\elabel{deltasqr}\\
&&-{\cal I}^{(1)}_{(l)}(s,x,U(s,x),\cdot) \nonumber
\end{eqnarray}
for each ${\cal I}\in\{{\cal L},{\cal J}\}$. Then, by applying the
Ito's formula (see, e.g., Theorem 1.14 and Theorem 1.16 in pages 6-9
of $\emptyset$ksendal and Sulem~\cite{okssul:appsto}) to the
function
\begin{eqnarray}
&&\zeta(\Delta V^{(1)}_{(l),\delta}(t,x))\equiv\mbox{Tr}\left(\Delta
V_{(l),\delta}^{(1)}(t,x)\right)e^{2\gamma t}\nonumber
\end{eqnarray}
for some $\gamma>0$, where Tr$(A)$ denotes the trace of the matrix
$A'A$ for a given matrix $A$, we have
\begin{eqnarray}
&&\zeta(\Delta V^{(1)}_{(l),\delta}(t,x))+\int_{t}^{T}
\mbox{Tr}\left(\Delta{\cal J}^{(1)}_{(l),\delta}(s,x,U)
-\Delta\bar{V}^{(1)}_{(l),\delta}(s,x) \right)e^{2\gamma s}ds
\elabel{pvitod}\\
&&+\int_{t}^{T}\int_{z_{j}>0}
\mbox{Tr}\left(\Delta\tilde{V}_{(l),\delta}^{(1)}(s^{-},x,z)
\right)e^{2\gamma s}\tilde{N}(\lambda ds,x,dz)
\nonumber\\
&=&2\int_{t}^{T}\left(-\gamma\mbox{Tr}\left(\Delta
V_{(l),\delta}^{(1)}(s,x)\right)+\left(\Delta
V_{(l),\delta}^{(1)}(s,x)\right)' \left(\Delta{\cal
L}^{(1)}_{(l),\delta}(s,x,U)\right)\right)e^{2\gamma s}ds-M(t)
\nonumber\\
&\leq&\left(-2\gamma+\frac{3K^{2}_{D}}{\hat{\gamma}}\right)
\int_{t}^{T}\mbox{Tr}\left(\Delta
V_{(l),\delta}^{(1)}(s,x)\right)e^{2\gamma s}ds
+\hat{\gamma}\int_{t}^{T}\left\|\Delta{\cal
L}^{(1)}_{(l),\delta}(s,x,U)\right\|^{2}e^{2\gamma s}ds-M(t)
\nonumber\\
&=&\hat{\gamma}\int_{t}^{T}\left\|\Delta{\cal
L}^{(1)}_{(l),\delta}(s,x,U)\right\|^{2}e^{2\gamma
s}ds-M_{\delta}(t) \nonumber
\end{eqnarray}
if, in the last equality, we take
\begin{eqnarray}
\hat{\gamma}=\frac{3K_{D}^{2}}{2\gamma}>0, \elabel{thgamma}
\end{eqnarray}
where $M_{\delta}(t)$ is a martingale of the following form,
\begin{eqnarray}
&&2\sum_{j=1}^{d}\int_{t}^{T}\left(\Delta
V^{(1)}_{(l),\delta}(s^{-},x)\right)'\left(\Delta({\cal
J}_{j})_{(l),\delta}^{(1)}(s^{-},x,U)
-\Delta(\bar{V}_{j})_{(l),\delta}^{(1)}(s^{-},x)\right)e^{2\gamma
s}dW_{j}(s)
\nonumber\\
&&-2\sum_{j=1}^{h}\int_{t}^{T} \int_{z_{j}>0}\left(\Delta
V^{(1)}_{(l),\delta}(s^{-},x)\right)'
\left((1/\delta)\tilde{V}_{j}(s^{-},x,z_{j})
+\tilde{V}^{(1)}_{j}(s^{-},x,z_{j})\right)e^{2\gamma
s}\tilde{N}(\lambda_{j}ds,x,dz_{j})
\nonumber\\
&&+2\sum_{j=1}^{h}\int_{t}^{T} \int_{z_{j}>0}\left(\Delta
V^{(1)}_{(l),\delta}(s^{-},x)\right)'
\left((1/\delta)(\tilde{V}_{j})_{(l),\delta}(s^{-},x,z_{j})\right)
e^{2\gamma s}(\tilde{N}_{j})_{(l),\delta}(\lambda_{j}ds,x,dz_{j}).
\nonumber
\end{eqnarray}
Now, it follows from Lemma 1.3 in pages 6-7 of Peskir and
Shiryaev~\cite{pesshi:optsto} that, for each $t\in[0,T]$ and
$\sigma>0$, there is a sequence of
$\{\delta_{n},n=1,2,...\}\subset[0,\sigma]$ such that
\begin{eqnarray}
&&E\left[ess\sup_{0\leq\delta\leq\sigma}\zeta(\Delta
V^{(1)}_{(l),\delta}(t,x))\right] \elabel{difsuco}\\
&=&E\left[ess\sup_{\{\delta_{n}:0\leq\delta_{n}\leq\sigma,n=1,2,...\}}
\zeta(\Delta V^{(1)}_{(l),\delta_{n}}(t,x))\right]\nonumber\\
&=&\lim_{n\rightarrow\infty}E\left[\zeta(\Delta V^{(1)}_{(l),
\delta_{n}}(t,x))\right]\nonumber\\
&\leq& \hat{\gamma}\lim_{n\rightarrow\infty}
E\left[\int_{t}^{T}\left\|\Delta{\cal
L}^{(1)}_{(l),\delta_{n}}(s,x,U)\right\|^{2}e^{2\gamma
s}ds\right]-\lim_{n\rightarrow\infty}E\left[M_{\delta_{n}}(t)\right]
\nonumber\\
&\leq&
\hat{\gamma}E\left[\int_{t}^{T}ess\sup_{0\leq\delta\leq\sigma}
\left\|\Delta{\cal
L}^{(1)}_{(l),\delta}(s,x,U)\right\|^{2}e^{2\gamma s}ds\right],
\nonumber
\end{eqnarray}
where ``esssup'' denotes the essential supremum and the first
inequality in \eq{difsuco} follows from \eq{pvitod}. So, by the
Lebesgue's dominated convergence theorem, we have
\begin{eqnarray}
&&\lim_{\sigma\rightarrow
0}E\left[ess\sup_{0\leq\delta\leq\sigma}\zeta(\Delta
V^{(1)}_{(l),\delta}(t,x))\right] \elabel{difsucI}\\
&\leq& \hat{\gamma}E\left[\int_{t}^{T}\lim_{\sigma\rightarrow
0}ess\sup_{0\leq\delta\leq\sigma} \left\|\Delta{\cal
L}^{(1)}_{(l),\delta}(s,x,U)\right\|^{2}e^{2\gamma s}ds\right],
\nonumber
\end{eqnarray}
since, due to the mean-value theorem and the conditions stated in
\eq{blipschitzo}, we have
\begin{eqnarray}
&&\left\|\Delta{\cal L}^{(1)}_{(l),\delta}(t,x,U)\right\| \leq
2K_{D}\left(\|U\|_{C^{k+1}(D,q)}+\|\bar{U}\|_{C^{m+1}(D,qd)}
+\left\|\tilde{U}\right\|_{\nu,1}\right). \nonumber
\end{eqnarray}
Then, by \eq{difsucI} and the Fatou's lemma, we know that, for any
sequence $\sigma_{n}$ satisfying $\sigma_{n}\rightarrow 0$ along
$n\in{\cal N}$, there is a subsequence ${\cal N}'\subset{\cal N}$
such that
\begin{eqnarray}
&&ess\sup_{0\leq\delta\leq\sigma_{n}}\zeta(\Delta
V^{(1)}_{(l),\delta}(t,x))\rightarrow 0\;\;\mbox{along}\;\;n\in{\cal
N}'\;\;\mbox{a.s.}, \elabel{zetazero}
\end{eqnarray}
which implies that the first-order derivative of $V$ in terms of
$x_{l}$ for each $l\in\{1,...,p\}$ exists and equals
$V_{(l)}^{(1)}(t,x)$ a.s. for each $t\in[0,T]$ and $x\in D$, and
moreover, it is $\{{\cal F}_{t}\}$-adapted. Next, it follows from
the similar proof as used in \eq{difsuco} that
\begin{eqnarray}
&&\lim_{\sigma\rightarrow
0}E\left[\int_{t}^{T}ess\sup_{0\leq\delta\leq\sigma}
\mbox{Tr}\left(\Delta{\cal J}^{(1)}_{(l),\delta}(s,x,U)
-\Delta\bar{V}^{(1)}_{(l),\delta}(s,x) \right)e^{2\gamma s}ds\right]
\elabel{difsucII}\\
&\leq& \hat{\gamma}E\left[\int_{t}^{T}\lim_{\sigma\rightarrow
0}ess\sup_{0\leq\delta\leq\sigma} \left\|\Delta{\cal
L}^{(1)}_{(l),\delta}(s,x,U)\right\|^{2}e^{2\gamma s}ds\right].
\nonumber
\end{eqnarray}
Thus, by \eq{zetazero} and \eq{difsucII}, we know that
\begin{eqnarray}
&&\lim_{\delta\rightarrow
0}\Delta\bar{V}^{(1)}_{(l),\delta}(t,x)=\lim_{\delta\rightarrow
0}\Delta{\cal J}^{(1)}_{(l),\delta}(t,x,U)=0\;\;\;\mbox{a.s.}
\nonumber
\end{eqnarray}
which implies that the first-order derivative of $\bar{V}$ with
respect to $x_{l}$ for each $l\in\{1,...,p\}$ exists and equals
$\bar{V}^{(1)}_{(l)}(t,x)$ a.s. for every $t\in[0,T]$ and $x\in D$,
and moreover, it is a $\{{\cal F}_{t}\}$-predictable process.
Similarly, we can get the conclusion for
$\tilde{V}_{(l)}^{(1)}(t,x,z)$ associated with each $l,t,x,z$.

Secondly, assuming that $(V^{(c-1)}(t,x),$ $\bar{V}^{(c-1)}(t,x),$
$\tilde{V}^{(c-1)}(t,x,z))$ corresponding to a given $(U(t,x),$
$\bar{U}(t,x),$ $\tilde{U}(t,x,z))$ $\in{\cal Q}^{2}_{{\cal
F}}([0,T])$ exists for any given $c\in\{1,2,...\}$. Then we can show
that
\begin{eqnarray}
&&\left(V^{(c)}(t,x), \bar{V}^{(c)}(t,x),
\tilde{V}^{(c)}(t,x,z)\right) \elabel{cthderivative}
\end{eqnarray}
exists for the given $c\in\{1,2,...\}$.

As a matter of fact, consider any fixed nonnegative integer numbers
$i_{1},...,i_{p}$ satisfying $i_{1}+...+i_{p}=c-1$ for the given
$c\in\{1,2,...\}$, each $f\in\{V,\bar{V}, \tilde{V}\}$, each
$l\in\{1,...,p\}$, and each small enough $\delta>0$,  let
\begin{eqnarray}
&&f_{(i_{1}...(i_{l}+1)...i_{p}),\delta}^{(c-1)}(t,x)\equiv
f_{(i_{1}...i_{p})}^{(c-1)}(t,x+\delta e_{l})\elabel{fdeltan}
\end{eqnarray}
correspond to ${\cal I}^{(c-1)}_{(i_{1}...i_{p})}(s,x+\delta
e_{l},U(s,x+\delta e_{l}),\cdot)$ with ${\cal I}\in\{{\cal L},{\cal
J}\}$ via \eq{sigmanV}, where the differential operators ${\cal L}$
and ${\cal J}$ are replaced by their $(c-1)$th-order partial
derivatives ${\cal L}^{(c-1)}_{(i_{1}...i_{p})}$ and ${\cal
J}^{(c-1)}_{(i_{1}...i_{p})}$. Similarly, let
$(V^{(c)}_{(i_{1}...(i_{l}+1)...i_{p})}(t,x),$
$\bar{V}^{(c)}_{(i_{1}...(i_{l}+1)...i_{p})}(t,x),$
$\tilde{V}^{(c)}_{(i_{1}...(i_{l}+1)...i_{p})}(t,x,z))$ be defined
through \eq{sigmanV} where ${\cal L}$ and ${\cal J}$ are replaced by
their $c$th-order partial derivatives ${\cal
L}^{(c)}_{i_{1}...(i_{l}+1)...i_{p}}$ and ${\cal
J}^{(c)}_{i_{1}...(i_{l}+1)...i_{p}}$ corresponding to a given $t,x,
U(t,x),$ $\bar{U}(t,x),$ $\tilde{U}(t,x,z)$. Moreover, define
\begin{eqnarray}
&&\Delta f^{(c)}_{(i_{1}...(i_{l}+1)...i_{p}),\delta}(t,x)
=\frac{f^{(c-1)}_{(i_{1}...(i_{l}+1)...i_{p}),\delta}(t,x)-f(t,x)}{\delta}
-f^{(c)}_{(i_{1}...(i_{l}+1)...i_{p})}(t,x) \elabel{ipartialII}
\end{eqnarray}
for each $f\in\{U,\bar{U},\tilde{U},V,\bar{V},\tilde{V}\}$ and let
\begin{eqnarray}
&&\Delta{\cal I}^{(c)}_{(i_{1}...(i_{l}+1)...i_{p}),\delta}(t,x,U)
\elabel{ideltasqr}\\
&&=\frac{1}{\delta}\left({\cal
I}^{(c-1)}_{(i_{1}...i_{p})}(t,x+\delta
e_{l},U(t,x+\delta e_{l}),\cdot)\right. \nonumber\\
&&\;\;\;\left.-{\cal
I}^{(c-1)}_{(i_{1}...i_{p})}(s,x,U(s,x),\cdot)\right)-{\cal
I}^{(c)}_{(i_{1}...(i_{l}+1)...i_{p})}(s,x,U(s,x),\cdot)
\nonumber
\end{eqnarray}
for ${\cal I}\in\{{\cal L},{\cal J}\}$. Then, by the It$\hat{o}$'s
formula and repeating the procedure as used in the second step, we
know that $(V^{(c)}_{(i_{1}...(i_{l}+1)...i_{p})}(t,x),$
$\bar{V}^{(c)}_{(i_{1}...(i_{l}+1)...i_{p})}(t,x),$
$\tilde{V}^{(c)}_{(i_{1}...(i_{l}+1)...i_{p})}(t,x,z))$ exist for
the given $c\in\{1,2,...\}$ and all $l\in\{1,...,p\}$, which implies
that the claim in \eq{cthderivative} is true.

Thirdly, it follows from the induction method with respect to
$c\in\{1,2,...\}$ that the claims stated in the lemma are true.
Hence we finish the proof of Lemma~\ref{differentiableV}. $\Box$
\end{proof}

\begin{lemma}\label{lemmathree}
Under the conditions of Theorem~\ref{bsdeyI}, all the claims in the
theorem are true.
\end{lemma}
\begin{proof}
Let $D_{{\cal F}}^{2}([0,T],R^{q})$ be the set of $R^{q}$-valued
$\{{\cal F}_{t}\}$-adapted and square integrable c\`adl\`ag
processes as in \eq{adaptednormI}. Moreover, for any given
$\gamma\in R$, define ${\cal M}^{D}_{\gamma}[0,T]$ to be the
following Banach space (see, e.g., the similar explanation as used
in Yong and Zhou~\cite{yonzho:stocon}, and Situ~\cite{sit:solbac})
\begin{eqnarray}
&&\;\;\;{\cal M}^{D}_{\gamma}[0,T]=D^{2}_{{\cal
F}}([0,T],R^{q})\times L^{2}_{{\cal F},p}([0,T],R^{q\times d})\times
L^{2}_{p}([0,T],R^{q\times h}) \elabel{combanach}
\end{eqnarray}
endowed with the norm: for any given $(U,\bar{U},\tilde{U})\in{\cal
M}^{D}_{\gamma}[0,T]$,
\begin{eqnarray}
\;\;\;\left\|(U,\bar{U},\tilde{U})\right\|^{2}_{{\cal
M}^{D}_{\gamma}} &\equiv&
\sum_{k=1}^{\infty}\xi(k)\left\|(U,\bar{U},\tilde{U})
\right\|^{2}_{{\cal M}^{D}_{\gamma,k}}, \elabel{comnorm}
\end{eqnarray}
where, without loss of generality, we assume that $m=k$ in
\eq{bspdef} and
\begin{eqnarray}
\;\;\;\;\left\|(U,\bar{U},\tilde{U})\right\|^{2}_{{\cal
M}^{D}_{\gamma,k}}&=&E\left[\sup_{0\leq t\leq
T}\left\|U(t)\right\|^{2}_{C^{k}(D,q)}e^{2\gamma t}\right]
+E\left[\int_{0}^{T}
\left\|\bar{U}(t)\right\|^{2}_{C^{k}(D,qd)}e^{2\gamma
t}dt\right]\elabel{kcomnorm}\\
&&+E\left[\int_{0}^{T}\left\|\tilde{U}(t)
\right\|^{2}_{\nu,k}e^{2\gamma t}dt\right]. \nonumber
\end{eqnarray}
In addition, through \eq{sigmanV}, we can define the following map,
\begin{eqnarray}
&&\Xi:(U(\cdot,x),\bar{U}(\cdot,x),\tilde{U}(\cdot,x,z))
\rightarrow(V(\cdot,x),\bar{V}(\cdot,x),
\tilde{V}(\cdot,x,z)).\nonumber
\end{eqnarray}
Then, based on the norm defined in \eq{comnorm}, we can show that
$\Xi$ forms a contraction mapping in ${\cal M}^{D}_{\gamma}[0,T]$.
As a matter of fact, consider $(U^{i}(\cdot,x),\bar{U}^{i}(\cdot,x),
\tilde{U}^{i}(\cdot,x,z))\in{\cal M}^{D}_{\gamma}[0,T]$ and
$(V^{i}(\cdot,x),$ $\bar{V}^{i}(\cdot,x),$
$\tilde{V}^{i}(\cdot,x,z))$ $=\Xi(U^{i}(\cdot,x)$,
$\bar{U}^{i}(\cdot,x), \tilde{U}^{i}(\cdot,x,z))$ with
$i\in\{1,2,...\}$, define
\begin{eqnarray}
&&\Delta f^{i}=f^{i+1}-f^{i}\;\;\;\mbox{with}\;\;\;
f\in\left\{U,\bar{U},\tilde{U},V,\bar{V},\tilde{V}\right\}\nonumber
\end{eqnarray}
and take
\begin{eqnarray}
&&\zeta(\Delta U^{i}(t,x))=\mbox{Tr}\left(\Delta
U^{i}(t,x)\right)e^{2\gamma t}. \elabel{firstzeta}
\end{eqnarray}
Then it follows from \eq{blipschitz} and the similar argument as
used in proving \eq{pvitod} that, for a $\gamma>0$ and each
$i\in\{2,3,...\}$,
\begin{eqnarray}
&&\zeta(\Delta U^{i}(t,x))+\int_{t}^{T}\mbox{Tr}\left(\Delta{\cal
J}(s,x,U^{i},U^{i-1}) -\Delta\bar{U}^{i}(s,x) \right)e^{2\gamma s}ds
\elabel{eupvitod}\\
&&+\int_{t}^{T}\int_{z>0}
\mbox{Tr}\left(\Delta\tilde{U}^{i}(s^{-},x,z) \right)e^{2\gamma
s}\tilde{N}(\lambda ds,x,dz)
\nonumber\\
&&\leq\hat{\gamma}\int_{t}^{T}\left\|\Delta{\cal
L}(s,x,U^{i},U^{i-1})\right\|^{2}e^{2\gamma s}ds-M_{(i)}(t)
\nonumber\\
&&\leq\hat{\gamma}K_{a}N^{i-1}(t)-M^{i}(t) \nonumber
\end{eqnarray}
where $K_{a}$ is some nonnegative constant depending only on
$K_{D}$, and for the last inequality, we have taken
\begin{eqnarray}
&&\hat{\gamma}=\frac{3K_{D}^{2}}{2\gamma}>0. \elabel{euthgamma}
\end{eqnarray}
Moreover, $N^{i-1}(t)$ appeared in \eq{eupvitod} is given by
\begin{eqnarray}
&&\;N^{i-1}(t)=\int_{t}^{T}\left(\left\|\Delta
U^{i-1}(s)\right\|^{2}_{C^{k}(D,q)}+\left\|\Delta\bar{U}^{i-1}(s)
\right\|^{2}_{C^{k}(D,qd)}+\left\|\Delta\tilde{U}^{i-1}(s)
\right\|^{2}_{\nu,k}\right)e^{2\gamma s}ds \elabel{euntt}
\end{eqnarray}
and $M^{i}(t)$ is a martingale of the following form,
\begin{eqnarray}
&&M^{i}(t)\elabel{eumsigman}\\
&=&-2\sum_{j=1}^{d}\int_{t}^{T}\left((\Delta
U^{i})(s^{-},x)\right)'\left(\Delta{\cal
J}_{j}(s^{-},x,U^{i},U^{i-1})
-(\Delta\bar{U}^{i})_{j}(s^{-},x)\right)
e^{2\gamma s}dW_{j}(s) \nonumber\\
&&+2\sum_{j=1}^{h}\int_{t}^{T} \int_{z_{j}>0}\left((\Delta
U^{i})_{j}(s^{-},x)\right)'
\left((\tilde{U}^{i})_{j}(s^{-},x,z_{j})\right) e^{2\gamma
s}\tilde{N}_{j}(\lambda_{j}ds,x,dz_{j}). \nonumber
\end{eqnarray}
Then, by \eq{eupvitod}-\eq{eumsigman} and the martingale properties
related to stochastic integral, we have
\begin{eqnarray}
&&E\left[\left\|\Delta U^{i}(t,x) \right\|^{2}e^{2\gamma
t}+\int_{t}^{T}\mbox{Tr}\left(\Delta{\cal J}(s,x,U^{i},U^{i-1})
-\Delta\bar{U}^{i}(s,x)
\right)e^{2\gamma s}ds\right. \elabel{vitodI}\\
&&\;\;\;\;\;\left.+\int_{t}^{T}\left\|\Delta\tilde{U}^{i}(s)
\right\|_{\nu,0}^{2}e^{2\gamma
s}ds\right]\nonumber\\
&&\leq\hat{\gamma}(T+1)K_{a}\left\|(\Delta
U^{i-1},\Delta\bar{U}^{i-1},\Delta\tilde{U}^{i-1})
\right\|^{2}_{{\cal M}^{D}_{\gamma,k}}. \nonumber
\end{eqnarray}

Next it follows from \eq{eumsigman}, the Burkholder-Davis-Gundy's
inequality (see, e.g., Theorem 48 in page 193 of
Protter~\cite{pro:stoint}) that
\begin{eqnarray}
&&\;\;\;E\left[\sup_{0\leq t\leq T}\left|M^{i}(t)\right|\right]
\elabel{vitodII}\\
&\leq&4\sum_{j=1}^{d}E\left[\sup_{0\leq t\leq
T}\left|\int_{0}^{t}\left(\Delta
U^{i}(s^{-},x)\right)'\right.\right.
\nonumber\\
&&\;\;\;\;\;\;\;\;\;\;\;\;\left.\left.\left(\Delta{\cal
J}_{j}(s^{-},x,U^{i},U^{i-1})-(\Delta\bar{U}^{i})_{j}
(s^{-},x)\right)e^{2\gamma s}dW_{j}(s)\right|\right]
\nonumber\\
&&+4\sum_{j=1}^{h}E\left[\sup_{0\leq t\leq T}\left|\int_{0}^{t}
\int_{z_{j}>0}\left(\Delta U^{i}(s^{-},x)\right)'
(\Delta\tilde{U}^{i})_{j}(s^{-},x,z_{j}) e^{2\gamma
s}\tilde{N}(\lambda_{j}ds,x,dz_{j})\right|\right]
\nonumber\\
&\leq&K_{b}\sum_{j=1}^{d}E\left[\left(\int_{0}^{T}\left\|\Delta
U^{i}(s,x)\right\|^{2}\left\|(\Delta{\cal
J}^{i})_{j}(s,x,U^{i},U^{i-1})-(\Delta\bar{U}^{i})_{j}
(s,x)\right\|^{2}e^{4\gamma s}ds\right)^{\frac{1}{2}}
\right]\nonumber\\
&&+K_{b}\sum_{j=1}^{h}E\left[\left(\int_{0}^{T}
\int_{z_{j}>0}\left\|\Delta U^{i}(s,x)\right\|^{2}
\left\|(\Delta\tilde{U}^{i})_{j}(s,x,z_{j})\right\|^{2} e^{4\gamma
s}\lambda_{j}\nu_{j}(x,dz_{j})ds)\right)^{\frac{1}{2}}
\right] \nonumber\\
&\leq&K_{b}E\left[\left(\sup_{0\leq t\leq T}\|\Delta
U^{i}(t,x)\|^{2}e^{2\gamma t}\right)^{\frac{1}{2}}\right.
\nonumber\\
&&\;\;\;\;\;\;\;\;\;\;\left(\sum_{j=1}^{d}\left(\int_{0}^{T}
\left\|\Delta{\cal
J}_{j}(s,x,U^{i},U^{i-1})-(\Delta\bar{U}^{i})_{j}(s,x)
\right\|^{2}e^{2\gamma s}ds\right)^{\frac{1}{2}}\right.
\nonumber\\
&&\;\;\;\;\;\;\;\;\left.\left.+\sum_{j=1}^{h}\left(\int_{0}^{T}
\int_{z_{j}>0}\left\|(\Delta\tilde{U}^{i})_{j}(s,x,z_{j})\right\|^{2}
e^{2\gamma s}\lambda_{j}\nu_{j}(dz_{j})ds)\right)^{\frac{1}{2}}
\right)\right] \nonumber\\
&\leq&\frac{1}{2}E\left[\sup_{0\leq t\leq T}\|\Delta
U^{i}(t,x)\|^{2}e^{2\gamma t}\right]\nonumber\\
&&+dK_{b}^{2}E\left[\left(\int_{0}^{T}\left\|\Delta{\cal
J}_{j}(s,x,U^{i},U^{i-1})-(\Delta\bar{U}^{i})_{j}(s,x)
\right\|^{2}e^{2\gamma s}ds\right)\right]\nonumber\\
&&+K_{b}^{2}E\left[\int_{0}^{T}
\left\|\Delta\tilde{U}^{i}(s)\right\|^{2}_{\nu,0}
e^{2\gamma s}ds\right] \nonumber\\
&\leq&\frac{1}{2}E\left[\sup_{0\leq t\leq T}\|\Delta
U^{i}(t,x)\|^{2}_{C^{0}(q)}e^{2\gamma
t}\right]+\hat{\gamma}(T+1)dK_{b}^{2}E\left[N^{i-1}(t)\right]
\nonumber
\end{eqnarray}
where $K_{b}$ is some nonnegative constant depending only on $K_{D}$
and $T$, and we have used \eq{vitodI} for the last inequality of
\eq{vitodII}. Thus it follows from \eq{eupvitod}-\eq{vitodII} that
\begin{eqnarray}
&&E\left[\sup_{0\leq t\leq T}\left\|\Delta
U^{i}(t)\right\|^{2}_{C^{0}(q)}e^{2\gamma
t}\right]\elabel{firstineu}\\
&&\leq 2\left(1+dK_{b}^{2}\right)\hat{\gamma}(T+1)\left\|(\Delta
U^{i-1},\Delta\bar{U}^{i-1},\Delta\tilde{U}^{i-1})
\right\|^{2}_{{\cal M}^{D}_{\gamma,k}}. \nonumber
\end{eqnarray}
Moreover it follows from \eq{eupvitod} and \eq{blipschitzo} that,
for $i\in\{3,4,...\}$,
\begin{eqnarray}
&&E\left[\int_{t}^{T}
\mbox{Tr}\left(\Delta\bar{U}^{i}(s,x)\right)e^{2\gamma s}ds\right]
\elabel{secondineu}\\
&\leq& 2E\left[\int_{t}^{T}\mbox{Tr}\left(\Delta{\cal
J}(s,x,U^{i},U^{i-1}) -\Delta\bar{U}^{i}(s,x)
\right)e^{2\gamma s}ds\right]\nonumber\\
&&+2E\left[\int_{t}^{T}\mbox{Tr}\left(\Delta{\cal
J}(s,x,U^{i},U^{i-1})\right)e^{2\gamma s}ds\right]
\nonumber\\
&\leq&2\hat{\gamma}K_{c}\left(\left\|(\Delta
U^{i-1},\Delta\bar{U}^{i-1},\Delta\tilde{U}^{i-1})
\right\|^{2}_{{\cal M}^{D}_{\gamma,k}}+\left\|(\Delta
U^{i-2},\Delta\bar{U}^{i-2},\Delta\tilde{U}^{i-2})
\right\|^{2}_{{\cal M}^{D}_{\gamma,k}}\right) \nonumber
\end{eqnarray}
where $K_{c}$ is some nonnegative constant depending only on $K_{D}$
and $T$. Thus it follows from \eq{eupvitod} and
\eq{firstineu}-\eq{secondineu} that
\begin{eqnarray}
&&\left\|(\Delta U^{i},\Delta\bar{U}^{i},\Delta\tilde{U}^{i})
\right\|^{2}_{{\cal M}^{D}_{\gamma,0}} \elabel{lastine}\\
&&\leq\hat{\gamma}K_{d}\left(\left\|(\Delta
U^{i-1},\Delta\bar{U}^{i-1},\Delta\tilde{U}^{i-1})
\right\|^{2}_{{\cal M}^{D}_{\gamma,k}}+\left\|(\Delta
U^{i-2},\Delta\bar{U}^{i-2},\Delta\tilde{U}^{i-2})
\right\|^{2}_{{\cal M}^{D}_{\gamma,k}}\right) \nonumber
\end{eqnarray}
where $K_{d}$ is some nonnegative constant depending only on $K_{D}$
and $T$.

Now, by Lemma~\ref{differentiableV} and the similar construction as
in \eq{firstzeta}, for each $c\in\{1,2,...\}$, we can define
\begin{eqnarray}
&&\zeta(\Delta U^{c,i}(t,x))\equiv\mbox{Tr}\left(\Delta
U^{c,i}(t,x)\right)e^{2\gamma t}, \elabel{secondzeta}
\end{eqnarray}
where
\begin{eqnarray}
&&\Delta U^{c,i}(t,x))=(\Delta U^{(0),i}(t,x)),\Delta
U^{(1),i}(t,x)),...,\Delta U^{(c),i}(t,x))'. \nonumber
\end{eqnarray}
Then it follows from the It$\hat{o}$'s formula and the similar
discussion for \eq{lastine} that
\begin{eqnarray}
&&\left\|(\Delta U^{i},\Delta\bar{U}^{i},\Delta\tilde{U}^{i})
\right\|^{2}_{{\cal M}^{D}_{\gamma,c}} \elabel{clastine}\\
&&\leq\hat{\gamma}K_{d}\left(\left\|(\Delta
U^{i-1},\Delta\bar{U}^{i-1},\Delta\tilde{U}^{i-1})
\right\|^{2}_{{\cal M}^{D}_{\gamma,k+c}}+\left\|(\Delta
U^{i-2},\Delta\bar{U}^{i-2},\Delta\tilde{U}^{i-2})
\right\|^{2}_{{\cal M}^{D}_{\gamma,k+c}}\right), \nonumber
\end{eqnarray}
which implies that
\begin{eqnarray}
&&\left\|(\Delta U^{i},\Delta\bar{U}^{i},\Delta\tilde{U}^{i})
\right\|^{2}_{{\cal M}^{D}_{\gamma}} \elabel{infinityine}\\
&&\leq\hat{\gamma}K_{f}\left(\left\|(\Delta
U^{i-1},\Delta\bar{U}^{i-1},\Delta\tilde{U}^{i-1})
\right\|^{2}_{{\cal M}^{D}_{\gamma}}+\left\|(\Delta
U^{i-2},\Delta\bar{U}^{i-2},\Delta\tilde{U}^{i-2})
\right\|^{2}_{{\cal M}^{D}_{\gamma}}\right). \nonumber
\end{eqnarray}
Since $(a^{2}+b^{2})^{1/2}\leq a+b$ for $a,b\geq 0$, we have
\begin{eqnarray}
&&\left\|(\Delta U^{i},\Delta\bar{U}^{i},\Delta\tilde{U}^{i})
\right\|_{{\cal M}^{D}_{\gamma}} \elabel{noninfinityine}\\
&\leq&\sqrt{\hat{\gamma}K_{f}}\left(\left\|(\Delta
U^{i-1},\Delta\bar{U}^{i-1},\Delta\tilde{U}^{i-1})\right\|_{{\cal
M}^{D}_{\gamma}}+\left\|(\Delta
U^{i-2},\Delta\bar{U}^{i-2},\Delta\tilde{U}^{i-2})\right\|_{{\cal
M}^{D}_{\gamma}}\right) \nonumber
\end{eqnarray}
where $K_{f}$ is some nonnegative constant depending only on $K_{D}$
and $T$. Therefore, by taking $\gamma>0$ large enough such that
$2\sqrt{\hat{\gamma}K_{f}}$ sufficiently small and by
\eq{infinityine}, we know that
\begin{eqnarray}
&&\sum_{i=3}^{\infty}\left\|(\Delta
U^{i},\Delta\bar{U}^{i},\Delta\tilde{U}^{i})\right\|_{{\cal
M}^{D}_{\gamma}} \elabel{seriescon}\\
&\leq&\frac{\sqrt{\hat{\gamma}
K_{f}}}{1-2\sqrt{\hat{\gamma}K_{f}}}\left(2\left\|(\Delta
U^{2},\Delta\bar{U}^{2},\Delta\tilde{U}^{2})\right\|_{{\cal
M}^{D}_{\gamma}}+\left\|(\Delta
U^{1},\Delta\bar{U}^{1},\Delta\tilde{U}^{1})\right\|_{{\cal
M}^{D}_{\gamma}}\right) \nonumber\\
&<&\infty.\nonumber
\end{eqnarray}
Thus, from \eq{seriescon}, we see that
$(U^{i},\bar{U}^{i},\tilde{U}^{i})$ with $i\in\{1,2,...\}$ forms a
Cauchy sequence in ${\cal M}^{D}_{\gamma}[0,T]$, which implies that
there is some $(U,\bar{U},\tilde{U})$ such that
\begin{eqnarray}
&&(U^{i},\bar{U}^{i},\tilde{U}^{i})\rightarrow
(U,\bar{U},\tilde{U})\;\;\mbox{as}\;\;i\rightarrow\infty\;\;
\mbox{in}\;\;{\cal M}^{D}_{\gamma}[0,T]. \elabel{finalcon}
\end{eqnarray}
Finally, by \eq{finalcon} and the similar procedure as used for
Theorem 5.2.1 in pages 68-71 of
$\emptyset$ksendal~\cite{oks:stodif}, we can finish the proof of
Lemma~\ref{lemmathree}. $\Box$
\end{proof}

\vskip 0.25cm \noindent {\bf Proof of Theorem~\ref{bsdeyI}}.

By combining Lemma~\ref{martindecom}-Lemma~\ref{lemmathree}, we can
reach a proof for Theorem~\ref{bsdeyI}. $\Box$

\section{B-SPDEs in Infinite Space Domain and under Random
Environment}\label{infran}

\subsection{B-SPDEs in the Infinite Space Domain}

First of all, for a given nonnegative integer $b$ and each
$n\in\{b+1,b+2,...\}$, define the following sequence of sets
\begin{eqnarray}
&& D_{n}=\left\{x\in R^{p}: b\leq\|x\|\leq n\right\}\elabel{dnset}
\end{eqnarray}
and let
\begin{eqnarray}
&&R_{b}^{p}=\left\{x\in R^{p}: \|x\|\geq b\right\}\elabel{infset}
\end{eqnarray}
Moreover, let $C^{\infty}(R_{b}^{p},q)$ be the Banach space endowed
with the following norm
\begin{eqnarray}
&&\|f\|_{C^{\infty}(R_{b}^{p},q)}\equiv\sum_{n=b+1}^{\infty}\xi(n+1)
\|f\|_{C^{\infty}(D_{n},q)} \elabel{irbspacen}
\end{eqnarray}
for each $f\in C^{\infty}(R_{b}^{p},q)$, and let
\begin{eqnarray}
&&\bar{{\cal Q}}^{2}_{{\cal F}}([0,T])\equiv \bar{L}^{2}_{{\cal
F}}([0,T],R^{q})\times \bar{L}^{2}_{{\cal F},p}([0,T],R^{qd})\times
\bar{L}^{2}_{p}([0,T],R^{q\times h}) \elabel{ibunisoI}
\end{eqnarray}
be the corresponding space defined in \eq{bunisoI} when the norm in
\eq{inftynorm} is replaced by the associated one given in
\eq{irbspacen}.
\begin{theorem}\label{infdn}
Assuming that there exists a nonnegative constant $K_{R_{b}^{p}}$
such that conditions \eq{blipschitz}-\eq{blipschitzI} are satisfied
when $K_{D}$ is replaced by $K_{R_{b}^{p}}$. Moreover, if ${\cal
L}(t,x,v,\cdot)$ and ${\cal J}(t,x,v,\cdot)$ are $\{{\cal
F}_{t}\}$-adapted for each $x\in R_{b}^{p}$, $z\in R_{+}^{h}$, and
any given $(v,\bar{v},\tilde{v})\in
C^{\infty}(R_{b}^{p},R^{q})\times C^{\infty}(R_{b}^{p},R^{q\times
d})\times \bar{L}^{2}_{\nu,\infty}(R_{b}^{p}\times
R^{h}_{+},R^{q\times h})$ with
\begin{eqnarray}
&&{\cal L}(\cdot,x,0,\cdot),{\cal J}(\cdot,x,0,\cdot)\in
\bar{L}^{2}_{{\cal F}}\left([0,T],R^{q}\right), \elabel{irblipic}
\end{eqnarray}
then the B-SPDE \eq{bspdef} has a unique andapted solution
satisfying,
\begin{eqnarray}
&&(V(\cdot,x),\bar{V}(\cdot,x),\tilde{V}(\cdot,x,z))\in\bar{{\cal
Q}}^{2}_{{\cal F}}([0,T]) \elabel{irbuniso}
\end{eqnarray}
where $V$ is a c\`adl\`ag process and the uniqueness is in the
sense: if there exists another solution
$(U(t,x),\bar{U}(t,x),\tilde{U}(t,x,z))$ as required, we have
\begin{eqnarray}
&&E\left[\int_{0}^{T}\left(\|U(t)-V(t)\|^{2}_{C^{\infty}
(R_{b}^{p},q)}+\|\bar{U}(t)-\bar{V}(t)\|^{2}_{C^{\infty}
(R_{b}^{p},qd)}+\|\tilde{U}(t)
-\tilde{V}(t)\|_{\nu,\infty}^{2}\right)dt\right]=0.
\elabel{iuniquesbI}
\end{eqnarray}
\end{theorem}
\begin{proof}
It follows from \eq{irblipic} and the similar argument used for
\eq{infinityine} in the proof of Lemma~\ref{lemmathree} that
\begin{eqnarray}
&&(U^{1}(\cdot,x),\bar{U}^{1}(\cdot,x),\tilde{U}^{1}(\cdot,x,z))
\in\bar{{\cal Q}}^{2}_{{\cal F}}([0,T]) \elabel{onirbuniso}
\end{eqnarray}
with $(U^{0},\bar{U}^{0},\tilde{U}^{0})=(0,0,0)$, where
$(U^{1},\bar{U}^{1},\tilde{U}^{1})$ is defined through \eq{sigmanV}
in Lemma~\ref{martindecom}. Then, over each $\{D_{n}\}$ with
$n\in\{b+1,b+2,...\}$, it follows from \eq{infinityine} in the proof
of Lemma~\ref{lemmathree} that
\begin{eqnarray}
&&\left\|(\Delta U^{i},\Delta\bar{U}^{i},\Delta\tilde{U}^{i})
\right\|^{2}_{{\cal M}^{R_{b}^{p}}_{\gamma}} \elabel{uinfinityine}\\
&&\leq\hat{\gamma}K_{g}\left(\left\|(\Delta
U^{i-1},\Delta\bar{U}^{i-1},\Delta\tilde{U}^{i-1})
\right\|^{2}_{{\cal M}^{R_{b}^{p}}_{\gamma}}+\left\|(\Delta
U^{i-2},\Delta\bar{U}^{i-2},\Delta\tilde{U}^{i-2})
\right\|^{2}_{{\cal M}^{R_{b}^{p}}_{\gamma}}\right) \nonumber
\end{eqnarray}
where $K_{g}$ is some nonnegative constant depending only on $T$ and
$R_{b}^{p}$. Then it follows from \eq{uinfinityine} that the
remaining proof for Theorem~\ref{infdn} can be conducted similarly
as in the proof of Theorem~\ref{bsdeyI}. Hence we finish the proof
of Theorem~\ref{infdn}. $\Box$
\end{proof}

\subsection{B-SPDEs under Random Environment}

Assuming that the random environment under consideration is
characterized by a $R^{p}$-valued Markov process $X(\cdot)$ with
continuous sample paths and its associated stopping time $\tau$ is
defined by
\begin{eqnarray}
&&\tau\equiv\inf\{t\in[0,T],\|X(t)\|<b\}.\elabel{stoptime}
\end{eqnarray}
Then the $q$-dimensional B-SPDEs with jumps under random environment
$X(\cdot)$ can be described as follows,
\begin{eqnarray}
\;\;V(t,X(t))&=&H(X(\tau))+\int_{t}^{\tau}{\cal
L}(s^{-},X(t),V,\cdot)ds
\elabel{rbspdef}\\
&&+\int_{t}^{\tau}\left({\cal J}(s^{-},X(t),V,\cdot)
-\bar{V}(s^{-},X(t)\right)dW(s) \nonumber\\
&&-\int_{t}^{\tau} \int_{z>0}\tilde{V}(s^{-},X(t),z,\cdot)
\tilde{N}(\lambda ds,X(t),dz). \nonumber
\end{eqnarray}
Moreover, define
\begin{eqnarray}
&&\hat{{\cal Q}}^{2}_{{\cal F}}([0,\tau])\equiv \hat{L}^{2}_{{\cal
F}}([0,\tau],R^{q})\times \hat{L}^{2}_{{\cal
F},p}([0,\tau],R^{qd})\times \hat{L}^{2}_{p}([0,\tau],R^{q\times h})
\elabel{rbunisoI}
\end{eqnarray}
be the corresponding space defined in \eq{bunisoI} when the norm in
\eq{inftynorm} is replaced by the following one,
\begin{eqnarray}
&&\|f(x)\|_{\infty}\equiv\sum_{i\in{\cal N}}\xi(i)\sup_{j\leq
i}\left\|f^{(j)}(x)\right\|, \elabel{rbspacen}
\end{eqnarray}
where $f^{(i)}(x)$ for each $x\in R_{b}^{p}$ is defined as in
\eq{difoperator}.
\begin{theorem}\label{thrandom}
Under the conditions as stated in Theorem~\ref{infdn}, the B-SPDE in
\eq{rbspdef} under random environment $X(\cdot)$ has a unique
adapted solution satisfying,
\begin{eqnarray}
&&(V(\cdot,X(\cdot)),\bar{V}(\cdot,X(\cdot)),\tilde{V}(\cdot,X(\cdot),z))
\in\hat{{\cal Q}}^{2}_{{\cal F}}([0,\tau]),
\elabel{rbuniso}
\end{eqnarray}
where $V$ is a c\`adl\`ag process and the uniqueness is in the
sense: if there exists another solution
$(U(t,X(t)),\bar{U}(t,X(t)),\tilde{U}(t,X(t)))$ as required, we have
\begin{eqnarray}
&&E\left[\int_{0}^{T}\left(\|U(t,X(t))-V(t,X(t))\|^{2}_{\infty}+
\|\bar{U}(t,X(t))-\bar{V}(t,X(t))\|^{2}_{\infty}\right.\right.
\nonumber\\
&&\;\;\;\;\;\;\left.\left.+\|\tilde{U}(t,X(t))
-\tilde{V}(t,X(t))\|_{\nu,\infty}^{2}\right)dt\right]=0.
\nonumber
\end{eqnarray}
\end{theorem}
\begin{proof}
First of all, it follows from the similar discussion as in
Situ~\cite{sit:solbac} that $\bar{Q}_{{\cal F}}([0,\tau])$ is a
Banach space. Then we know that all of the claims stated in
Theorem~\ref{infdn} are true over the space $\bar{Q}_{{\cal
F}}([0,\tau])$, which imply that the claims in the current theorem
are true. $\Box$
\end{proof}
\begin{example}
The solution $V(t,x,\cdot)$ to the B-SPDE in \eq{bspdef} is
described by random surfaces and the solution $V(t,X(t))$ to the
B-SPDE under random environment in \eq{rbspdef} is represented by
random paths, which are shown in Figure~\ref{twouserregion}
presented in the Introduction.
\end{example}

\section{An Illustrative Example in Finance}\label{expm}

In this section, we consider a financial market consisting of two
assets and an external random factor. One asset is supposed to be a
risk-free account whose price $S_{0}(t)$ is subject to the following
ordinary differential equation,
\begin{eqnarray}
&&\left\{\begin{array}{ll}
       dS_{0}(t)=rS_{0}(t)dt,\\
       S_{0}(0)=1,
       \end{array}
\right. \elabel{bankaset}
\end{eqnarray}
where the interest rate $r$ is a nonnegative constant. Another asset
is stock whose price process $S(t)$ satisfies the following SDE for
each $t\in[0,T]$,
\begin{eqnarray}
&&\left\{\begin{array}{ll}
       dS(t)=S(t)\beta(Y(t))dt+S(t)\sigma(Y(t))dW_{1}(t),\\
       S(0)=s>0,
       \end{array}
\right. \elabel{stockassetm}
\end{eqnarray}
where the random factor $Y(t)$ with $t\in[0,T]$ satisfies
\begin{eqnarray}
&&\left\{\begin{array}{ll}
       dY(t)=c(Y(t))dt+d(Y(t))\left(\rho dW_{1}(t)
       +\sqrt{1-\rho^{2}}dW_{2}(t)\right),\\
       Y(0)=y\in R
       \end{array}
\right. \elabel{stockassetmI}
\end{eqnarray}
with $\rho\in(-1,1)$. Moreover, we suppose that the market
coefficients $f=\beta,\sigma,c,d$ satisfy the standard global
Lipschitz and linear growth conditions and $\sigma(y)\geq\kappa>0$
for all $y\in R$ and some positive constant $\kappa$.

Beginning at $t=0$ with an initial endowment $x\in R_{+}$, an
investor invests at any time $t>0$ in the risky and riskless asset.
The present value of the amounts invested are denoted, respectively,
by $\pi_{0}(t)$ and $\pi_{1}(t)$, and then the present value of the
investor's aggregate investment is given by
$X^{\pi}(t)=\pi_{0}(t)+\pi_{1}(t)$, which satisfies (see, e.g.,
Musiela and Zairphopoulou~\cite{muszar:stopar})
\begin{eqnarray}
&&dX^{\pi}(t)=\sigma(t)\pi(t)\cdot(\lambda(t)dt+dW(t))
\elabel{wealtheq}
\end{eqnarray}
where $``\cdot"$ denotes the inner product,
$\pi(t)=(\pi_{0}(t),\pi_{1}(t))$, $dW=(dW_{1},dW_{2})'$, and
\begin{eqnarray}
&&\lambda(t)=\left(\frac{\beta(Y(t))-r}{\sigma(y(t))},0\right)'.
\elabel{lambdaex}
\end{eqnarray}
Moreover, for a given constant $b\geq 1$, let the following $\tau$
be the bankruptcy time for the investor,
\begin{eqnarray}
&&\tau=\inf\{t>0,X^{\pi}(t)<b\}.\elabel{bankr}
\end{eqnarray}

One objective to study the above financial system is to find the
optimal portfolio choice based on maximal expected utility of
terminal wealth over all admissible strategies (see, e.g.,
Merton~\cite{mer:lifpor}), i.e., to solve the following stochastic
dynamic optimization problem,
\begin{eqnarray}
&&V(t,x)=\sup_{{\cal
A}_{\tau}}E_{P}\left[\left.u_{\tau}(X^{\pi}(\tau))\right|{\cal
F}_{t},X^{\pi}(t)=x\right] \elabel{vvvxx}
\end{eqnarray}
where ${\cal A}_{\tau}$ denotes the set of all admissible strategies
$\pi$: $\pi(t)$ is self-financing and $\{{\cal
F}_{t}\}$-progressively measurable, satisfying
\begin{eqnarray}
&&E\left[\int_{0}^{\tau}(\pi(s)\sigma(s))^{2}ds\right]<\infty\;\;
\mbox{and}\;\;X^{\pi}(t)\geq 0\;\;\mbox{with}\;\;t\in[0,\tau],
\nonumber
\end{eqnarray}
and the utility is taken to be the following constant relative risk
aversion (CRRA) case:
\begin{eqnarray}
&&u_{\tau}(x)=\frac{x^{\gamma}}{\gamma}\;\;\mbox{with}\;\;0<\gamma<1,
\gamma\neq 0.\nonumber
\end{eqnarray}
Then it follows from the discussions in Musiela and
Zariphopoulou~\cite{muszar:stopar},
$\emptyset$ksendal~\cite{oks:stodif}, $\emptyset$ksendal and
Sulem~\cite{okssul:appsto} that the value function process defined
in \eq{vvvxx} should satisfy the following B-SPDE,
\begin{eqnarray}
V(t,x)&=&u_{\tau}(X(\tau))+\int_{t}^{\tau}\frac
{\left(V_{x}(s,x)\lambda(t)+\sigma(t)(\sigma(t))^{+}\bar{V}_{x}(s,x)
\right)^{2}}{2V_{xx}(s,x)}ds \elabel{fbspde}\\
&&+\int_{t}^{\tau}(\bar{V}(s,x))'dW(s)\nonumber
\end{eqnarray}
where $\sigma(t)=(\sigma(Y(t)),0)'$,
$(\sigma(t))^{+}=(1/\sigma(Y(t)),0)$, and $dW=(dW_{1},dW_{2})'$.

As pointed out in Musiela and Zariphopoulou~\cite{muszar:stopar},
the B-SPDE in \eq{fbspde} is newly derived and belongs to a class of
strongly nonlinear B-SPDEs (see, e.g., the related discussion in
Lions and Souganidis~\cite{liosou:notaux}). However, based on
Theorem~\ref{infdn} in the previous section of the current paper and
the discussion in Musiela and Zariphopoulou~\cite{muszar:stopar}, we
can show that there exists a unique adapted solution, before a
random bankruptcy time (i.e., over $[0,\tau]$), to the B-SPDE in
\eq{fbspde} over the class of functions satisfying the conditions
required by Theorem~\ref{infdn}. In fact, based on the discussions
in Musiela and Zariphopoulou~\cite{muszar:stopar}, Glosten {\em et
al.}~\cite{glojag:relexp}, we have the following observation that
there is a pair of $V$ and $\bar{V}$ satisfying \eq{fbspde}, i.e.,
\begin{eqnarray}
&&V(t,x)=\frac{1}{\gamma}x^{\gamma}f(t,Y(t))^{\delta}\elabel{vsolution}
\end{eqnarray}
where $f$ is a solution of the following partial differential
equation
\begin{eqnarray}
&&f_{t}+\frac{1}{2}d^{2}(y)f_{yy}+\left(c(y)
+\frac{\rho\gamma\lambda(y)d(y)}{1-\gamma}\right)f_{y}
+\frac{\gamma\lambda^{2}(y)f}{2\delta(1-\gamma)}=0 \nonumber
\end{eqnarray}
with $f(\tau,y)=1$ and
$\delta=(1-\gamma)/(1-\gamma+\rho^{2}\gamma)$, and
\begin{eqnarray}
&&\bar{V}_{1}(t,x)=\frac{\rho\delta}{\gamma}d(Y(t))
f_{y}(t,Y(t))f(t,Y(t))^{\delta-1}
\elabel{barvI}\\
&&\bar{V}_{2}(t,x)=\frac{\delta(1-\rho^{2})^{1/2}}{\gamma}
x^{\gamma}d(Y(t))f_{y}(t,Y(t))f(t,Y(t))^{\delta-1} \elabel{barVII}
\end{eqnarray}
Thus it follows from \eq{vsolution}-\eq{barVII} that $(V,\bar{V})$
is a solution to \eq{fbspde} such that the conditions in
Theorem~\ref{infdn} are satisfied over $[0,\tau]$ and hence it is
the unique adapted solution to \eq{fbspde} over $[0,\tau]$.
Moreover, it follows from Theorem~\ref{thrandom} in the previous
section and the discussion in Musiela and
Zariphopoulou~\cite{muszar:stopar} that, for each $t\leq s\leq\tau$,
the optimal feedback portfolio process is given as follows,
\begin{eqnarray}
&&\pi^{*}(s,X(s))=-\frac{(\sigma(s))^{+}\left(V_{x}(s,X(s))\lambda(s)
+\sigma(s)(\sigma(s))^{+}\bar{V}_{x}(s,X(s))\right)}{V_{xx}(s,X(s))}.
\nonumber
\end{eqnarray}

\vskip 0.6cm \noindent{\bf Acknowledgement} This project is
supported by National Natural Science Foundation of China under
grant No. 10971249.

\end{document}